\documentclass[12pt]{article}
\date{}
\usepackage[ansinew]{inputenc}
\usepackage{array}
\usepackage{subcaption}
\usepackage{color}
\usepackage{ragged2e}
\usepackage{amsmath,amsthm,amscd}
\usepackage{amsxtra}
\usepackage{amsfonts,mathrsfs}
\usepackage{amstext}
\usepackage{amssymb}
\usepackage{cite}
\usepackage{graphics}
\usepackage{latexsym}
\usepackage{amsmath,amsfonts,amssymb,amsthm,epsfig,epstopdf,titling,url,array}
\usepackage{cite}
\usepackage[a4paper, total={6.3in, 9.4in}]{geometry}
\usepackage{titlesec} % Allows customization of titles
\numberwithin{equation}{section}

\newtheorem{cor}{Corollary}[section]
\theoremstyle{definition}
\newtheorem{definition}{Definition}[section]

\newtheorem{rem}{Remark}[section]
\newtheorem{nota}{Notations}[section]

\newtheorem{thm}{Theorem}[section]
\newtheorem{Lemma}{Lemma}[section]
\newtheorem{prop}{Proposition}[section]
\providecommand{\U}[1]{\protect\rule{.1in}{.1in}}
\begin{document}
\begin{center}
 ~\\\vspace{1.9cm}
	{\Large  \bf{ Characterization of non-singular hyperplanes of $\mathcal{H}\left(s,q^2\right)$} in $\mathrm{P G}\left(s, q^2\right)$ }\\	%\\
  ~\\\vspace{0.3cm}
  {\large \bf Stuti Mohanty \qquad Bikramaditya Sahu}\\%\footnote{  Corresponding Author}\qquad Bikramaditya Sahu}\\
  %$^1$Department of Mathematics, NIT Rourkela.\\

	\end{center}
     \vspace{0.3cm}
    \begin{center}
        {\bf Abstract}
    \end{center}
		%\noindent
        In this paper, we present a combinatorial characterization of the hyperplanes associated with non-singular hermitian varieties $\mathcal H\left(s, q^2\right)$ in the projective space $\mathrm{PG}\left(s,q^2\right)$ where $s\geq3$ and $q>2$.
 By analyzing the intersection numbers of hyperplanes with points and co-dimension $2$ subspaces, we establish necessary and sufficient conditions for a hyperplane to be part of the hermitian variety. This approach extends previous characterizations of hermitian varieties based on intersection properties, providing a purely combinatorial method for identifying their hyperplanes.
	 %In this paper, we characterize the set of hyperplanes of a hermitian varieties $\mathcal H\left(s, q^2\right)$  using their intersection numbers with points and co-dimension 2 subspaces.
    %The present article deals with the approximation properties of the linear positive operators, including general-Appell polynomials. We establish some results for the convergence of the operators and their order of approximation with the help of the modulus of continuity, Lipschitz class, and Vrornovskaja-type theorem. This article concludes with some numerical examples, their graphical representation, and some remarks. 
 \\

	\\
 
	
	\section{Introduction}
    Projective geometry over finite fields, especially the space $\mathrm{P G}\left(s, q^2\right)$ where $q$ is a prime power, gives us an organized but complicated way to look into the deep links between geometry, combinatorics, and algebra. Within this framework, hermitian varieties emerge as important research issues, distinguished by their attractive combination of algebraic form and geometric structure. These varieties, defined by non-degenerate sesquilinear forms over the finite field $\mathrm{GF}(q^2)$, add the Frobenius automorphism $ x \mapsto x^q$ to classical quadrics. The resulting forms yield homogeneous equations of degree $q+1$, whose solutions in $\mathrm{P G}\left(s, q^2\right)$ constitute varieties with distinctive combinatorial properties and symmetries. These unique features make hermitian varieties a central topic in the study of finite classical polar spaces and their use in areas like combinatorial designs and coding theory.\\

   % These forms result in homogeneous equations of degree $q+1$, whose solutions in $\mathrm{P G}\left(s, q^2\right)$ generate variety with unique combinatorial features and symmetry.\\ 

    Segre, Tallini, and others established the framework for the study of hermitian varieties by characterizing their intersections with subspaces, singularities, and embedding features. Among the characterizations of different subspaces of the polar spaces, the characterization of the hyperplane draws the attention of many researchers. In 2010, \cite{bb3} provides a characterization of non-singular quadrics and hermitian varieties based on intersection numbers with hyperplanes and co-dimension $2$ spaces. In 2013, \cite{bb8} Butler provides a characterization of the planes that intersect a non-degenerate quadric of $\mathrm{PG}\left(4,q\right)$ in a conic. Then in 2020, two new studies (\cite{bb1}, \cite{bb2}) provide characterizations of elliptic and hyperbolic hyperplanes in $\mathrm{PG}\left(4,q\right)$. After two years, in 2022 \cite{bb6}, Winter and Schillewaert generalize the above to $\mathrm{P G}\left(2s, q\right)$. Sahu \cite{bb9} gives a characterization of the planes that meet a hyperbolic quadric of $\mathrm{PG}\left(3,q\right)$ in a conic in 2022, and then he \cite{bb6} generalizes this by characterizing parabolic hyperplanes of hyperbolic and elliptic quadrics in $\mathrm{PG}\left(2s+1,q\right)$.\\

\begin{nota}
    We use the notations listed below:
    \begin{itemize}
        \item  A non-singular hermitian variety in $\mathrm{PG}\left(s,q^2\right)$ is denoted by $\mathcal H\left(s, q^2\right)$.
   \item  The symbols $p\mathcal{H}$ or $L\mathcal{H}$ stand for a cone having a vertex of a point $p$ or a line $ L$ and a base of a non-singular hermitian variety $\mathcal{H}$.
    \end{itemize}
\end{nota}
    %We use the following notations:\\
    %A cone with vertex a point $p$ or a line $L$ and base a non-singular hermitian variety $\mathcal{H}$ is denoted by $p\mathcal{H}$ or $L\mathcal{H}$, respectively.\\

In this research, we give a characterization of non-singular hyperplanes in $\mathrm{PG}\left(s,q^2\right)$ with respect to $\mathcal H\left(s, q^2\right)$, and our major result is as follows:

    \begin{thm} \label{main thm}
        {\em Let $\Omega$ be a non-empty family of hyperplanes of $\mathrm{PG}\left(s,q^2\right)$ that satisfies the following properties:
\begin{itemize}
    \item[$(I.)$] Every point of $\mathrm{P G}\left(s, q^2\right)$ lies in $\displaystyle\frac{q^s\left(q^{s-1}-(-1)^{s-1}\right)}{q+1}$ or $\displaystyle\frac{q^{s-1}\left(q^s-(-1)^s\right)}{q+1}$ hyperplanes of $\Omega$.
    \item [$(II.)$] Every subspace of $\mathrm{P G}\left(s, q^2\right)$ of co-dimension 2, which is contained in a hyperplane of $\Omega$, is contained in at least $q^2-q$ and at most $q^2$ hyperplanes of $\Omega$.
\end{itemize}
Then $\Omega$ is the set of hyperplanes of $\mathrm{PG}\left(s, q^2\right)$ meeting $\mathcal H\left(s, q^2\right)$ in $\mathcal H\left(s-1, q^2\right)$.}
    \end{thm}

    Quasi-quadrics are defined in \cite{bb4}. There is a similar definition of quasi-hermitian variety.
    \begin{definition}
A {\it quasi-hermitian variety} is a point set in $\mathrm{PG}\left(s, q^2\right)$ having the same intersection number with respect to hyperplanes of a non-singular hermitian variety in $\mathrm{PG}\left(s, q^2\right)$. Non-singular hermitian varieties are examples of quasi-hermitian varieties in $\mathrm{PG}\left(s, q^2\right)$. There are quasi-hermitian varieties that are not non-singular hermitian varieties. For more details, we refer to \cite{bb10}.
\end{definition}
\begin{rem} \label{remart}
There are two types of points with respect to $\Omega$ according to Property (I) of Theorem \ref{main thm}. In this paper, we first prove that the set of all first-type points, say $\mathbf{H}$, forms a quasi-hermitian variety. Only the first property is not enough to show that $\Omega$ is defined with respect to a hermitian variety. Therefore, it is necessary to put an extra condition on $\Omega$ to classify the set of non-tangent hyperplanes with respect to a non-singular hermitian variety in $\mathrm{PG}\left(s, q^2\right)$. However, we can't put the condition (similar to \cite[Theorem 1.2]{bb5}) as ``(II)$^\star$:  Every subspace of $\mathrm{PG}\left(s, q^2\right)$ of co-dimension $2$ which is contained in a hyperplane of $\Omega$ is contained in at least $q^2-q$ hyperplanes of $\Omega$". This is because there are quasi-hermitian varieties (which are not non-singular hermitian varieties) and a set of hyperplanes, satisfying both the properties of Theorem \ref{main thm}. This follows from the following description, one can refer to \cite{bb10} for all the terminologies and justifications.
 
 Consider a non-singular hermitian variety $\mathcal{H}\left(3,q^2\right)$ in $\mathrm{PG}\left(3, q^2\right)$. Now consider the set $S=\left(\mathcal{H}\left(3,q^2\right)\setminus p\mathcal{H}\left(1,q^2\right)\right)\cup p\mathcal{H}^\prime$, called the ``pivoted set"\cite{bb10}, where $\mathcal{H}^\prime$ is a quasi-hermitian variety in $\mathcal{H}\left(1,q^2\right)$. Then $S$ is a quasi-hermitian variety in $\mathcal{H}\left(3,q^2\right)$. The set of hyperplanes/planes, say $\Pi$, meeting $S$ in $\left|\mathrm{\mathcal{H}}\left(2, q^2\right)\right|$ points satisfy both properties (I) and (II)$^\star$. In particular, there are lines (or co-dimension 2 subspaces) which are contained in $0, ~q^2-q$ and $q^2+1$ planes (or hyperplanes) of $\Pi$. This allows us to redefine Property (II)$^\star$ to Property (II) given in Theorem \ref{main thm}, which differs from the main theorem of \cite[Theorem 1.2]{bb5}.
\end{rem}

We use the following version of \cite[Theorem 1.6]{bb3}, to prove Theorem \ref{main thm}.
\begin{prop} \label{use} \cite[Theorem 1.6]{bb3} 
{\em If a set of points $\mathbf{H}$ in $\mathrm{PG}\left(s,q^2\right)$, where $s\geq3$ and $q>2$ shares the same intersection numbers with hyperplanes and co-dimension 2 subspaces as a known non-singular hermitian variety $\mathcal H\left(s, q^2\right)$, then $\mathbf{H}$ must be $\mathcal H\left(s, q^2\right)$ itself.}

%A set of points $\mathbf{H}$ in $\mathrm{PG\eft(s,q^2\right)$, $n\geq3$, $q>2$ is the point set of a non-singular hermitian variety $\mathcal H\left(s, q^2\right)$ if it has the same intersection numbers with regard to the hyperplanes and co-dimension $2$ subspaces of $\mathcal H\left(s, q^2\right)$.
%If a point set $\mathbf{H}$ in $\mathrm{PG\eft(s,q^2\right)$, $n\geq3$, $q>2$, has the same intersection numbers with respect to the hyperplanes and codimension $2$ subpaces as a non-singular hermitian variety $\mathcal H\left(s, q^2\right)$, then $\mathbf{H}$ is the point set of non-singular hermitian variety $\mathcal H\left(s, q^2\right)$.
    
\end{prop}

We shall first demonstrate that the collection of hyperplanes that satisfy Property (I) of Theorem \ref{main thm} constitutes a quasi-hermitian variety.

   \begin{prop} \label{quasi prop}
       {\em Let $\Omega$ be a non-empty collection of hyperplanes of $\mathrm{PG}\left(s,q^2\right)$ such that 
       \begin{itemize}
    \item[$ (I.)$] Each point of $\mathrm{P G}\left(s, q^2\right)$ is contained in $\displaystyle\frac{q^s\left(q^{s-1}-(-1)^{s-1}\right)}{q+1}$ or $\displaystyle\frac{q^{s-1}\left(q^s-(-1)^s\right)}{q+1}$ hyperplanes of $\Omega$.
\end{itemize}
Then, a quasi-hermitian variety $\mathbf {H}$ is formed by the set of points contained in $\displaystyle\frac{q^s\left(q^{s-1}-(-1)^{s-1}\right)}{q+1}$ hyperplanes of $\Omega$ and $\Omega$ is the collection of all hyperplanes that meet $\mathbf{H}$ in $\left|\mathrm{\mathcal{H}}\left(s-1, q^2\right)\right|$ points.}
%Then the set of points contained in $\frac{q^s\left(q^{s-1}-(-1)^{s-1}\right)}{q+1}$ hyperplanes of $\Omega$ forms a quasi hermitian variety $\mathbf {H}$ and $\Omega$ is the set of all hyperplanes that meet $\mathbf{H}$ in $|\mathcal H\left(s, q^2\right)|$ points.
%$\Omega$ is the set of all hyperplanes of $\mathrm{PG}\left(s, q^2\right)$ meeting $\mathcal H\left(s, q^2\right)$ in $\mathcal H\left(s-1, q^2\right)$.
   
   \end{prop}
   
   %=================================================================
   \section{Preliminaries}
   The number of points in $\mathcal H\left(s, q^2\right)$ is $\displaystyle\frac{\left(q^{s+1}+(-1)^s\right)\left(q^s-(-1)^s\right)}{q^2-1}$. A hyperplane of $\mathrm{PG} \left(s, q^2\right)$ intersects $\mathcal H\left(s, q^2\right)$  in a cone $ p\mathrm{\mathcal{H}}(s-2, q^2)$ with a point $p$ as vertex or in a non-singular hermitian variety $\mathcal H\left(s-1, q^2\right)$   called as {\it tangent} or {\it non-tangent} hyperplanes, respectively. A co-dimension $2$ subspace  intersects $\mathcal H\left(s, q^2\right)$ either in a non-singular hermitian variety $\mathcal H\left(s-2, q^2\right)$, in a cone $ p\mathcal{H}(s-3, q^2)$ %with a point $\mathrm{P}$ as vertex 
   or $\mathrm{L \mathcal{H}}\left(s-4, q^2\right)$.\\ %whose vertex is a line $\mathrm{L}$,

   The hyperplane intersection numbers are 
   \begin{equation*}
       %|\mathcal H\left(s-1, q^2\right)| = 
       m_1 = \left|\mathrm{\mathcal{H}}\left(s-1, q^2\right)\right|,
       %\frac{\left(q^s-(-1)^s\right)\left(q^{s-1}+(-1)^s\right)}{q^2-1},
   \end{equation*}
   or
   \begin{equation*}
       %|P\mathcal H\left(s-2, q^2\right)| 
       m_2 = 1+ q^2\left|\mathrm{\mathcal{H}}\left(s-2, q^2\right)\right|,
       %\frac{q^2\left(q^{s-1}+(-1)^s\right)\left(q^{s-2}-(-1)^s\right)}{q^2-1},
   \end{equation*}
   and with co-dimension $2$ subspaces are 
       %\item cone $P\mathcal H\left(s-2, q^2\right)$ i.e. $\Omega$ is $1+\frac{q^2\left(q^{s-1}+(-1)^s\right)\left(q^{s-2}-(-1)^s\right)}{q^2-1}$
       %\item $\mathcal H\left(s-1, q^2\right)$ i.e. $\frac{\left(q^s-(-1)^s\right)\left(q^{s-1}+(-1)^s\right)}{q^2-1}$
   %\end{itemize}
    
   \begin{equation*}
      %\left|\mathcal H\left(s-2, q^2\right)\right|
      n_1 = \left|\mathrm{\mathcal{H}}\left(s-2, q^2\right)\right|,
      %\frac{\left(q^{s-1}+(-1)^s\right)\left(q^{s-2}-(-1)^s\right)}{q^2-1},
   \end{equation*}
   \begin{equation*}
      %\left| \mathrm{P\mathcal{H}}(s-3, q^2)\right|
      n_2 =1+q^2 \left|\mathrm{\mathcal{H}}\left(s-3, q^2\right)\right|,
      %\frac{\left.q^{s-3}+(-1)^s\right)\left(q^{s-2}-(-1)^s\right)}{q^2-1},
   \end{equation*}
   or
   \begin{equation*}
      % \left|\mathrm{L\mathcal{ H}}\left(s-4, q^2\right)\right|
      n_3 =1+q^2+q^4 \left|\mathrm{\mathcal{H}}\left(s-4, q^2\right)\right|.
      %\frac{\left(q^{s-3}+(-1)^s\right)\left(q^{s-4}-(-1)^s\right)}{q^2-1}.
   \end{equation*}
  \justifying One can also refer to \cite[Section 2.2]{bb3},  for the aforementioned intersection number.
  % Through a point of $\mathcal H\left(s, q^2\right)$, there is a unique tangent hyperplane and hence,
  There are in total $|\mathcal H\left(s, q^2\right)|$ tangent hyperplanes and $\displaystyle\frac{q^s\left(q^{s+1}+(-1)^s\right)}{q+1}$ non-tangent hyperplanes. For a detailed study on the hermitian variety, one can refer to \cite{bb7}.
   
   %$\displaystyle\frac{\left(q^{s+1}+(-1)^s\right)\left(q^s-(-1)^s\right)}{q^2-1}$ tangent hyperplanes and $\displaystyle\frac{q^s\left(q^{s+1}+(-1)^s\right)}{q+1}$ non-tangent hyperplanes. For a detailed study about the hermitian variety, one can refer to \cite{bb7}.

   %==================================================================================================================================================

   \begin{prop}\label{qua}
      {\em Let $\Omega$ be the collection of all hyperplanes that meet $\mathcal H\left(s, q^2\right)$ in $|\mathcal H\left(s-1, q^2\right)|$ points. Then the collection $\Omega$ satisfies both the Properties of Theorem \ref{main thm}.} 
   \end{prop}

   \begin{proof}
   The automorphism group of the unitary polarity acts transitively on the set of points on $\mathcal H\left(s, q^2\right)$ as well as it acts transitively on the set of points of $\mathrm{PG} \left(s, q^2\right) \setminus \mathcal H\left(s, q^2\right)$. It follows that a point on and off the hermitian variety has a constant number of hyperplanes through it.
   %It follows that the number of hyperplanes through a point on and off the hermitian variety is constant.
   
   Note that $|\Omega| = \displaystyle\frac{q^s\left(q^{s+1}+(-1)^s\right)}{q+1} $. We first count the total number of hyperplanes through a point of $\mathcal H\left(s, q^2\right)$. Let $h$ be the number of hyperplanes of $\Omega$ through a point of $\mathcal H\left(s, q^2\right)$.
      % Let $x$ be a point of $\mathrm{PG} \left(s, q^2\right)$ and let  $x$. Let $x \in \mathcal H\left(s, q^2\right)$.
       We count the
number of point-hyperplanes order pairs

\begin{equation*}
    \{(x,\Phi)~|~ x \in \mathcal H\left(s, q^2\right), \Phi\in\Omega~ \text{and}~x\in\Phi\},
\end{equation*}
in two different ways, we obtain
\begin{equation*}
\left|\mathrm{\mathcal{H}}\left(s, q^2\right)\right|h= |\Omega| \left|\mathrm{\mathcal{H}}\left(s-1, q^2\right)\right|,
%\frac{q^s\left(q^{s+1}+(-1)^{s}\right)}{q+1} \left(\frac{q^{2 s}-1}{q^{2}-1}\right) 
\end{equation*}
this gives $\displaystyle h = \frac{q^s\left(q^{s-1}-(-1)^{s-1}\right)}{q+1}$.\\
 Next, we count the total number of hyperplanes through a point outside of $\mathcal H\left(s, q^2\right)$. Let $m$ be the number of hyperplanes of $\Omega$ through a point outside of $\mathcal H\left(s, q^2\right)$. We count the
number of point-hyperplanes order pairs
%Now let $x\notin \mathcal H\left(s, q^2\right)$. We count in two different ways the number of order pairs of point-hyperplanes
\begin{equation*}
    \{(x,\Phi)~|~ x \in \mathrm{P G}\left(s, q^2\right) \setminus \mathcal H\left(s, q^2\right), \Phi\in\Omega~ \text{and}~x\in\Phi\},
\end{equation*}
 in two different ways, we obtain 
\begin{equation*}
\left(\left|\mathrm{PG} \left(s, q^2\right)\right| - \left|\mathcal H\left(s, q^2\right)\right|\right)m = |\Omega| \left(\left(\frac{q^{2s}-1}{q^{2}-1}\right) - \left|\mathcal H\left(s-1, q^2\right)\right|\right),
%\frac{q^s\left(q^{s+1}+(-1)^{s}\right)}{q+1} \left(\left(\frac{q^{2s}-1}{q^{2}-1}\right) - \left(\frac{q^{2s-2}-1}{q^{2}-1}\right)\right),
\end{equation*}
this gives $\displaystyle m = \frac{q^{s-1}\left(q^s-(-1)^s\right)}{q+1}$. %This verifies (I).

Let $\Psi$ be a co-dimension 2 subspace intersecting $\mathcal{H}(s,q^2)$ in $n_i$ points %( where  $c_1 = |\mathcal H\left(s-2, q^2\right)|$, $c_2 = |p\mathrm{\mathcal{H}}(s-3, q^2)|$ and $c_3 = |\mathrm{L \mathcal{H}}\left(s-4, q^2\right)|$ ) 
and denote $s_i$ as the corresponding number of hyperplanes of $\Omega$ containing $\Psi$. Now calculating the total points of $\mathcal H\left(s, q^2\right)$ in $\mathrm{P G}\left(s, q^2\right)$ by taking all the hyperplanes containing $\Psi$ and then counting points of $\mathcal H\left(s, q^2\right)$ in each such hyperplanes, we have the following 
%Let $\Psi$ be a codimension 2 subspace intersecting $\mathcal{H}(s,q^2)$ in $\mathcal H\left(s-2, q^2\right)$ points and denote $s$ as the corresponding number of hyperplanes of $\Omega$ containing $\Psi$. Now counting the total number of points of $\mathcal H\left(s, q^2\right)$ in $\mathrm{P G}\left(s, q^2\right)$ by taking all the hyperplanes containing $\Psi$ and then counting points of $\mathcal H\left(s, q^2\right)$ in each such hyperplanes, 
%Let $\Psi$ be a codimension 2 subspace intersecting $\mathcal{H}(s,q^2)$ in $c_i$ points and let $s_i$ be the number of hyperplanes of $\Omega$ containing $\Psi$ for $i=1,2,3$.  Note that the
%total number of points of $\mathcal H\left(s, q^2\right)$ in $\mathrm{P G}\left(s, q^2\right)$  can be obtained by taking all
%the hyperplanes containing $\Psi$ and then counting points of $\mathcal H\left(s, q^2\right)$ in each such hyperplanes. . So, we obtain the following 
\begin{equation*}
    s_i(\left|\mathcal H\left(s-1, q^2\right)\right| - n_i) + (q^2+1-s_i)(| p\mathrm{\mathcal{H}}(s-3, q^2)| - n_i) = \left|\mathcal{H}\left(s, q^2\right)\right| - n_i.
\end{equation*}
%we get $s=0$. Now if $\Psi$ intersecting $\mathcal{H}(s,q^2)$ in $ \mathrm{P\mathcal{H}}(s-3, q^2)$ or $\mathrm{L \mathcal{H}}\left(s-4, q^2\right)$ points, then solving in a similar way as above we get $s = q^2-q$ or $s=q^2$. We get a codimension $2$ subspace of $\mathrm{P G}\left(s, q^2\right)$ is contained in $0, q^2-q, q^2$ hyperplanes of $\Omega$.
For $i = 1,2$ and $3$, solving the equation in $s_i$ for different $n_i$, we find that a co-dimension $2$ subspace of $\mathrm{P G}\left(s, q^2\right)$ is contained in $0,~ q^2-q$ or $q^2$ hyperplanes of $\Omega$.

    \end{proof}

   %=================================================================
        
\section{Proof of Proposition \ref{quasi prop}}
%\begin{proof}
    A point of $\mathrm{PG} \left(s, q^2\right)$ is said to be black or white if it is contained in $\displaystyle \frac{q^s\left(q^{s-1}-(-1)^{s-1}\right)}{q+1}$ or $\displaystyle\frac{q^{s-1}\left(q^s-(-1)^s\right)}{q+1}$
hyperplanes of $\Omega$, respectively.
Let us denote $b$ and $w$ the number of black and white points of $\mathrm{P G}\left(s, q^2\right)$, respectively. Then,
\begin{equation}\label{e1}
    b+w=\frac{q^{2 s+2}-1}{q^2-1}.
\end{equation}
Counting the given set 
\begin{equation*}
    \{(x,\Phi)~|~ x \text{~is a point of }\mathrm{P G}\left(s, q^2\right), \Phi\in\Omega ~\text{and}~x\in\Phi\},
\end{equation*}
in two ways, we obtain
\begin{equation}\label{e2}
\quad b  \left[\frac{q^s\left(q^{s-1}-(-1)^{s-1}\right)}{q+1}\right]+w  \left[\frac{q^{s-1}\left(q^s-(-1)^s\right)}{q+1}\right]=\left|\Omega\right|\left(\frac{q^{2 s}-1}{q^{2}-1}\right).
\end{equation}
%======================================================================(1.2)
Putting the value of $w$ from Equation \eqref{e1} into Equation \eqref{e2}, we obtain
\begin{equation*}
b \left[\frac{q^s\left(q^{s-1}-(-1)^{s-1}\right)}{q+1}\right]+\left(\frac{q^{2 s+2}-1}{q^2-1}-b\right)\left[\frac{q^{s-1}\left(q^s-(-1)^s\right)}{q+1}\right]=|\Omega|\left(\frac{q^{2 s}-1}{q^2-1}\right).
\end{equation*}
%Now solving the above, we get
%\begin{equation*}
    %\Rightarrow b\left[\frac{q^s\left(q^{s-1}-(-1)^{s-1}\right)}{q+1}-\frac{q^{s-1}\left(q^s-(-1)^s\right)}{q+1}\right]+\left(\frac{q^{2s+2}-1}{q^2-1}\right)\left[\frac{q^{s-1}\left(q^s-(-1)^s\right)}{q+1}\right]=|\Omega|\left(\frac{q^{2s}-1}{q^2-1}\right)
%\end{equation*}
On simplifying, we have 
\begin{equation}\label{e3}
    (-1)^s q^{s-1} b+ q^{s-1} \left(\frac{q^{2 s+2}-1}{q^2-1}\right) \left(\frac{q^s-(-1)^s}{q+1}\right)=|\Omega|\left(\frac{q^{2 s}-1}{q^2-1}\right).
\end{equation} 
%=====================================================================(1.3)
Since, $\displaystyle \gcd\left(\left(\frac{q^{2 s}-1}{q^2-1}\right), q^{s-1} \right)= 1$, then from Equation \eqref{e3}, we get
\begin{equation*}
    q^{s-1}\bigg| \mid \Omega \mid.
\end{equation*}
As $\displaystyle \frac{q^{2s}-1}{q^2-1}=\frac{\left(q^s+(-1)^s\right)\left(q^s-(-1)^s\right)}{q^2-1}$, %=\left(\frac{q^s-(-1)^s}{q^2-1}\right)\left(q^s+(-1)^s\right)\.$
then again from Equation \eqref{e3}
\begin{equation}\label{e4}
\begin{cases}\displaystyle\left.\left(\frac{q^s-1}{q^2-1}\right) \right\rvert\, b, & \text { if $s$ is even },\\ \\ \displaystyle\left.\left(\frac{q^s+1}{q+1}\right) \right\rvert\, b, &\text {if $s$ is odd }\cdot \end{cases}
    %\left.\left(\frac{q^s-(-1)^s}{q^2-1}\right) \right\rvert\, b
\end{equation} 
From Equation \eqref{e1}, $\quad \displaystyle 0 \leqslant b \leqslant \frac{q^{2 s+2}-1}{q^2-1}$. Since
\begin{equation*}   
\frac{q^{2 s+2}-1}{q^2-1} = 1+ q^2\left(\frac{q^s+1}{q+1}\right)\left(\frac{q^s-1}{q-1}\right) , 
\end{equation*}
 from Equation \eqref{e4}, we have

\begin{equation}\label{e16}
     b \leqslant \begin{cases}q^2\left(q^s+1\right)\displaystyle\left(\frac{q^s-1}{q^2-1}\right), & \text { if $s$ is even }, \\ \\ q^2\displaystyle\left(\frac{q^s-1}{q-1}\right)\left(\frac{q^s+1}{q+1}\right), &\text { if $s$ is odd }.\end{cases}
\end{equation}
%\noindent This is obtained by dividing the
%\left(\frac{q^{2 s+2}-1}{q^2-1}\right) \text { by }\left(\frac{q^s-1}{q^2-1}\right) \text{or by } \left(\frac{q^s+1}{q+1}\right)
%$ according to $s$ even or odd and then taking the term that gives the maximum possible integer value.
Now we write Equation \eqref{e3} as
\begin{equation}\label{e5}
    b=(-1)^s\left[\frac{|\Omega|}{q^{s-1}}\left(\frac{q^{2s}-1}{q^2-1}\right)-\left(\frac{q^{2 s+2}-1}{q^2-1}\right)\left(\frac{q^s-(-1)^s}{q+1}\right)\right].
\end{equation}

For simplification, we define $|\Omega _1|:= \displaystyle \frac{|\Omega|}{q^{s-1}}$. Now we prove the following important lemma.

%================================================================================

\begin{Lemma}\label{hyp}
   {\em  \begin{equation*}
     |\Omega _1|= \begin{cases}\displaystyle \frac{q^{s+2}-1}{q+1}-(q-1)+t, & \text { if $s$ is even } \quad \\ \\
     \displaystyle \frac{q^{s+2}+1}{q+1}-\left(q^2-q+1\right)+t , & \text { if $s$ is odd } \quad \end{cases} 
     %q^s \frac{\left(q^{s+1}+1\right)}{q+1}  \quad\text {, $s$-even } \\
%\frac{q^s\left(q^{s+1}-1\right)}{q+1}, \quad\text { $s$-odd } 
\end{equation*}
for some $1 \leqslant t \leqslant q^2$.}
\end{Lemma}

\begin{proof}
We consider two cases. For $s$-even, by Equation \eqref{e16},
\begin{equation*}
    0 \leqslant b \leqslant q^2\left(q^s+1\right)\left(\frac{q^s-1}{q^2-1}\right).
\end{equation*}
From Equation \eqref{e5}, we have
\begin{equation*}
    0 \leqslant |\Omega _1|\left(\frac{q^{2 s}-1}{q^2-1}\right)-\left(\frac{q^{2 s+2}-1}{q^2-1}\right)\left(\frac{q^s-1}{q+1}\right) \leqslant q^2\left(q^s+1\right)\left(\frac{q^s-1}{q^2-1}\right).
\end{equation*}
Equivalently,
%It follows that
%\begin{equation*}
   % 0 \leqslant \frac{|\Omega|}{q^{s-1}}\left({q^s+1}\right)-\left(\frac{q^{2 s+2}-1}{q+1}\right) \leqslant q^2\left(q^s+1\right)
%\end{equation*}

\begin{equation*}
  \frac{q^{2 s+2}-1}{(q+1)\left(q^s+1\right)} \leqslant |\Omega _1| \leqslant q^2+\frac{q^{2 s+2}-1}{(q+1)\left(q^s+1\right)}.
\end{equation*}
Now, since 
\begin{equation*}
    \frac{q^{2 s+2}-1}{(q+1)\left(q^s+1\right)} = \frac{q^{s+2}-1}{q+1}-(q-1) + \frac{q^2-1}{(q+1)\left(q^s+1\right)},
\end{equation*}
and $|\Omega _1|$ is an integer, we have
\begin{equation*}
   \frac{q^{s+2}-1}{q+1}-(q-1) < |\Omega _1| \leqslant q^2+\frac{q^{s+2}-1}{q+1}-(q-1).
\end{equation*}
From above, we can write
\begin{equation*}
    |\Omega _1|=\frac{q^{s+2}-1}{q+1}-(q-1)+t, \quad \text { for some } 1 \leqslant t \leqslant q^2.
\end{equation*}
For $s$-odd, by Equation \eqref{e16},
\begin{equation*}
    0 \leqslant b \leqslant q^2\left(q^s-1\right)\left(\frac{q^s+1}{q^2-1}\right).
\end{equation*} 
Now again from Equation \eqref{e5},
\begin{equation*}
    0 \leqslant\left(\frac{q^{2 s+2}-1}{q^2-1}\right)\left(\frac{q^s+1}{q+1}\right)-|\Omega _1|\left(\frac{q^{2 s}-1}{q^2-1}\right) \leqslant q^2\left(\frac{q^s-1}{q-1}\right)\left(\frac{q^s+1}{q+1}\right) .
\end{equation*}
Now, since
\begin{equation*}
    \frac{q^{2 s+2}-1}{(q+1)\left(q^s-1\right)} = \frac{q^{s+2}+1}{q+1}+(q-1) + \frac{q^2+1}{(q+1)\left(q^s-1\right)},
\end{equation*}
and, $|\Omega _1|$ is an integer, we have

%Similarly, solving the above inequality as in even case, we get 
\begin{equation*}
\frac{q^{s+2}+1}{q+1}-(q^2-q-1) < |\Omega _1| \leqslant q^2+\frac{q^{s+2}+1}{q+1}-(q^2-q+1).
\end{equation*}
Then we can write
\begin{equation*}\label{e6}
    \left|\Omega_1\right|=\frac{q^{s+2}+1}{q+1}-\left(q^2-q+1\right)+t , \quad \text { for some } 1 \leqslant t \leqslant q^2.
\end{equation*}
This proves the lemma.
\end{proof}

The following lemma tells us that in each hyperplane of $\Omega $ the number of black points is constant.

%===========================================================================

\begin{Lemma}\label{bl-hyp}
    {\em The number of black points in a hyperplane $\Phi$ of $\Omega$ is given by \begin{equation*}
    b_\Phi=\frac{(-1)^s}{q^{s-1}}\left[(|\Omega|-1)\left(\frac{q^{2 s-2}-1}{q^2-1}\right)-\left(\frac{q^{2 s}-1}{q^2-1}\right)\left(\frac{q^{s-1}\left(q^s-(-1)^s\right)}{q+1}-1\right)\right].
\end{equation*} }
\end{Lemma}

\begin{proof}
Let $\Phi$ be a hyperplane in $\Omega$. Denote $b_\Phi$ as the number of black points in $\Phi$. Then there are $\displaystyle \frac{q^{2 s}-1}{q^2-1}-b_\Phi$ white points in $\Phi$.\\ %Then\\
%\begin{equation} \label{e8}
   % b_\Phi+w_\Phi=\frac{q^{2 s}-1}{q^2-1}.
%\end{equation}
Counting the given set 
\begin{equation*}
   \{(x, \Sigma) \mid x \text{~ is a point of~} \mathrm{PG}\left(s, q^2\right), \Sigma \in \Omega,~ \Sigma \neq \Phi ~\text{and} ~x \in \Sigma \cap\Phi\},
\end{equation*}
in two ways, we have
%\begin{equation} \label{e9}
    %b_\Phi\left[\frac{q^s\left(q^{s-1}-(-1)^{s-1}\right)}{q+1}-1\right]+\omega_\Phi\left[\frac{q^{s-1}\left(q^s-(-1)^s\right)}{q+1}-1\right]=(|\Omega|-1)\left(\frac{q^{2 s-2}-1}{q^2-1}\right)
%\end{equation}

%Now putting the value of $w_\Phi$ from eq \eqref{e8} into \eqref{e9},

\begin{equation*}
    b_\Phi\left[\frac{q^s\left(q^{s-1}-(-1)^{s-1}\right)}{q+1}-1\right]+\left(\frac{q^{2 s}-1}{q^2-1}-b_\Phi\right)\left[\frac{q^{s-1}\left(q^s-(-1)^s\right)}{q+1}-1\right]=(|\Omega|-1)\left(\frac{q^{2 s-2}-1}{q^2-1}\right).%\right.
\end{equation*}
%{\em i.e.,}
%\begin{multline*}
    %b_\Phi\left[\frac{q^s\left(q^{s-1}-(-1)^{s-1}\right)}{q+1}- \frac{q^{s-1}\left(q^s-(-1)^s\right)}{q+1}\right]+\left(\frac{q^{2 s}-1}{q^2-1}\right)\left[\frac{q^{s-1}\left(q^s-(-1)^s\right)}{q+1}-1\right]\\=(|\Omega|-1)\left(\frac{q^{2 s-2}-1}{q^2-1}\right)%\right.
%\end{multline*}
%\begin{equation*}
    %\Rightarrow (-1)^{s}q^{s-1}b_\Phi+ \left(\frac{q^{2 s}-1}{q^2-1}\right)\left[\frac{q^{s-1}\left(q^s-(-1)^s\right)}{q+1}-1\right]=(|\Omega|-1)\left(\frac{q^{2 s-2}-1}{q^2-1}\right)%\right.
%\end{equation*}
Putting the value of $|\Omega| = q^{s-1}|\Omega_1|$ from Lemma \ref{hyp} and simplifying, we find the desired expression for $b_\Phi$.
%On simplifying, we get
%\begin{equation*}
    %b_\Phi=\frac{(-1)^s}{q^{s-1}}\left[(|\Omega|-1)\left(\frac{q^{2 s-2}-1}{q^2-1}\right)-\left(\frac{q^{2 s}-1}{q^2-1}\right)\left[\frac{q^{s-1}\left(q^s-(-1)^s\right)}{q+1}-1\right]\right].
%\end{equation*}
\end{proof}

Let $b_\Phi$ be the number of black points in a hyperplane of $\Omega$ given in Lemma \ref{bl-hyp}. The subsequent lemma serves as the central finding of this paper.
% The following lemma is the main result of the paper.

%============================================================================

\begin{Lemma}
    $\displaystyle |\Omega|= \displaystyle \frac{q^s\left(q^{s+1}+(-1)^s\right)}{q+1}$.
\end{Lemma}
\begin{proof}
Counting the point-hyperplane incident pairs,
\begin{equation*}
    \left\{(x, \Phi) \mid x \text{~is a black point in~} \mathrm{PG}\left(s, q^2\right), \Phi \in \Omega ~\text{and}~ x \in \Phi\right\},
\end{equation*}
in two ways, we get

\begin{equation*}
    b \frac{q^s\left(q^{s-1}-(-1)^{s-1}\right)}{q+1}=|\Omega| b_\Phi.
\end{equation*}
Putting the value of $b_\Phi$ from Lemma \ref{bl-hyp}, we obtain
\begin{equation*}
    { b } \frac{q^s\left(q^{s-1}-(-1)^{s-1}\right)}{q+1}=|\Omega| \frac{(-1)^s}{q^{s-1}}\left[(|\Omega|-1)\left(\frac{q^{2s-2}-1}{q^2-1}\right)-\left(\frac{q^{2s}-1}{q^2-1}\right)\left(\frac{q^{s-1}\left(q^s-(-1)^s\right)}{q+1}-1\right)\right],
\end{equation*}
that is,
\begin{equation*}
     \quad b=\frac{|\Omega|(q+1)}{q^s\left(q^{s-1}-(-1)^{s-1}\right)} \times\frac{(-1)^s}{q^{s-1}}\left[(|\Omega|-1)\left(\frac{q^{2 s-2}-1}{q^2-1}\right)-\left(\frac{q^{2s}-1}{q^2-1}\right)\left(\frac{q^{s-1}\left(q^s-(-1)^s\right)}{q+1}-1\right)\right].
\end{equation*}
Also, from Equation \eqref{e5}, we have
\begin{equation*}
    b=(-1)^s\left[\frac{|\Omega|}{q^{s-1}} \left(\frac{q^{2 s}-1}{q^2-1}\right)-\left(\frac{q^{2 s+2}-1}{q^2-1}\right)\left(\frac{q^s-(-1)^s}{q+1}\right)\right].
    \end{equation*}
Equating the above two values of $b$, we get
\begin{multline}\label{e7}
    \frac{|\Omega|}{q^{s-1}}\left[(\mid \Omega \mid-1)\left(\frac{q^{2 s-2}-1}{q^2-1}\right)-\left(\frac{q^{2 s}-1}{q^2-1}\right)\left(\frac{q^{s-1}\left(q^s-(-1)^s\right)}{q+1}-1\right)\right] \\
 =\frac{q^s\left(q^{s-1}-(-1)^{s-1}\right)}{q+1}\left[\frac{\mid \Omega \mid}{q^{s-1}}\left(\frac{q^{2 s}-1}{q^2-1}\right)-\left(\frac{q^{2 s+2}-1}{q^2-1}\right)\left(\frac{q^s-(-1)^s}{q+1}\right)\right].
\end{multline}
Now we consider two cases.\\
\\
\noindent\textbf{Case I: \underline{$s$-even}}\\
Recall that $\displaystyle |\Omega| = \frac{|\Omega_1|}{q^{s-1}}$. %Then Equation \eqref{e7} may be written as
%\begin{multline}\label{e8}
    %|\Omega_1|\left[\left(q^{s-1}\left|\Omega_1\right|-1\right)\left(\frac{q^{2 s-2}-1}{q^2-1}\right)-\left(\frac{q^{2 s}-1}{q^2-1}\right)\left(\frac{q^{s-1}\left(q^s-1\right)}{q+1}-1\right)\right]\\=\frac{q^s\left(q^{s-1}+1\right)}{q+1}\left[\left|\Omega_1\right|\left(\frac{q^{2s}-1}{q^2-1}\right)-\left(\frac{q^{2 s+2}-1}{q^2-1}\right)\left(\frac{q^s-1}{q+1}\right)\right].
%\end{multline}
Now RHS of Equation \eqref{e7} is
\begin{equation*}
    \frac{q^s\left(q^{s-1}+1\right)}{q+1}\left[\left|\Omega_1\right|\left(\frac{q^{2 s}-1}{q^2-1}\right)-\left(\frac{q^{2 s+2}-1}{q^2-1}\right)\left(\frac{q^s-1}{q+1}\right)\right],
\end{equation*}
putting $\displaystyle |\Omega_1|=\frac{q^{s+2}-1}{q+1}-(q-1)+t \quad$ (from Lemma \ref{hyp}) and simplifying the RHS of Equation \eqref{e7}, we have
%\begin{equation*}
    %\frac{q^s\left(q^{s-1}+1\right)}{(q+1)\left(q^2-1\right)}\left[\left(\frac{q^{s+2}-1}{q+1}-(q-1)+t\right)\left(q^{2 s}-1\right)-\left(q^{2 s+2}-1\right)\left(\frac{q^s-1}{q+1}\right)\right],
%\end{equation*}
  % or, equivalently 
  %\notag  &\implies& \frac{q^s\left(q^{s-1}+1\right)}{(q+1)^2\left(q^2-1\right)}\left[\left(q^{s+2}-1-\left(q^2-1\right)+t(q+1)\right)\left(q^{2 s}-1\right)-\left(q^{2 s+2}-1\right)\left(q^s-1\right)\right]\\
   %\notag  &\implies& \frac{q^s\left(q^{s-1}+1\right)}{(q+1)^2\left(q^2-1\right)}\left[t\left(q^{2s}-1\right)\left(q+1\right) - \left(q^s-1\right)\left(q^2-1\right)\right]\\
\begin{equation*} \frac{q^s\left(q^{s-1}+1\right)\left(q^s-1\right)}{\left(q^2-1\right)(q+1)}\left[t\left(q^s+1\right)-(q-1)\right].
\end{equation*}
%after simplification, we have  
%\begin{equation}
   %\implies \frac{q^s\left(q^{s-1}+1\right)\left(q^2-1\right)}{\left(q^2-1\right)(q+1)}\left[t\left(q^s+1\right)-(q-1)\right]
%\end{equation}
Now LHS of Equation \eqref{e7} is
\begin{equation*}
 \left|\Omega_1\right|\left[\left(q^{s-1}\left|\Omega_1\right|-1\right)\left(\frac{q^{2 s-2}-1}{q^2-1}\right)-\left(\frac{q^{2 s}-1}{q^2-1}\right)\left(\frac{q^{s-1}\left(q^s-1\right)}{q+1}-1\right)\right],
   \end{equation*}
   putting $\displaystyle |\Omega_1|=\frac{q^{s+2}-1}{q+1}-(q-1)+t \quad$ (from Lemma \ref{hyp}) and simplifying the LHS of Equation \eqref{e7}, we have %LHS of Equation \eqref{e8} becomes\\
  %\begin{multline*}
%= \left(\frac{q^{s+2}-1}{q+1}-(q-1)+t\right) \left[ \left(q^{s-1} \left(\frac{q^{s+2}-1}{q+1}-(q-1)+t\right)-1\right) \left(\frac{q^{2s-2}-1}{q^2-1}\right) \right. \\
%\left. - \left(\frac{q^{2s}-1}{q^2-1}\right) \left(\frac{q^{s-1}(q^s-1)}{q+1}-1\right) \right],
%\end{multline*}
%equivalently
   %\begin{multline*}
%= \left(\frac{q^{s+2}-q^2+t(q+1)}{(q+1)^2(q^2-1)}\right) \left[ \left(q^{2s+1}-q^{s+1}+tq^s+tq^{s-1}-q-1\right) \left(q^{2s-2}-1\right) \right. \\
%\left. - \left(q^{2s}-1\right) \left(q^{2s-1}-q^{s-1}-q-1\right) \right]
%\end{multline*}
 %\begin{equation*}
%\left(\frac{q^{s+2}-q^2+t(q+1)}{(q+1)^2\left(q^2-1\right)}\right)\left[t\left((q+1)(q^{3 n-3}-q^{s-1})\right)+\left(q^2-1\right)(q^{2 s-2}+q^{s-1})\right].
%\end{equation*}
%After simplifying the above term, LHS of Equation \eqref{e8} is 
\begin{equation*}
  \frac{q^{s-1}\left(q^{s-1}+1\right)}{\left(q^2-1\right)(q+1)}\left[q^2\left(q^{s}-1\right)+t(q+1)\right]\left[t\left(q^{s-1}-1\right)+(q-1)\right].
\end{equation*}
Now, equating the LHS and RHS and simplifying, we get
\begin{multline*}
      \frac{q^{s-1}\left(q^{s-1}+1\right)}{\left(q^2-1\right)(q+1)}\left[q^2\left(q^{s}-1\right)+t(q+1)\right]\left[t\left(q^{s-1}-1\right)+(q-1)\right] = \\ \frac{q^s\left(q^{s-1}+1\right)\left(q^s-1\right)}{\left(q^2-1\right)(q+1)}\left[t\left(q^s+1\right)-(q-1)\right],
\end{multline*}
%that is,
%\begin{equation*}
%\left[q^2\left(q^{s}-1\right)+t(q+1)\right]\left[t\left(q^{s-1}-1\right)+(q-1)\right] = q(q^s-1)\left[t\left(q^s+1\right)-(q-1)\right].
%\end{equation*}
and again simplifying, we obtain
\begin{equation}\label{e10}
    q\left(q^s-1\right)\left(t-q+1\right)=\left[t\left(q^{s-1}-1\right)+(q-1)\right]t.
\end{equation}
Denote $\displaystyle k_s = \frac{q^{s}-1}{q-1}$, then we  write Equation \eqref{e10} as 
\begin{equation}\label{e11}
    qk_s\left(t-q+1\right)=t^2k_{s-1}+t.
\end{equation}
As RHS of Equation \eqref{e11} is greater than or equals to $1$, it follows that $t \geqslant q$. Clearly $t =q$ is a solution of the above equation. So we  factor Equation \eqref{e11} as
\begin{equation*}
   (t-q)(k_{s-1}t-k_s(q-1)) = 0. 
\end{equation*}
which implies $ t = q$ or $\displaystyle t = \frac{k_s(q-1)}{k_{s-1}}$. Since $\displaystyle\frac{k_s(q-1)}{k_{s-1}}$ is not an integer for $s \geqslant 4$,
we have $t =q$.\\
\\
%$q^{s-1}\nmid t$ as $0\leqslant t\leqslant q^2$ also $t\neq 0$
%$$\implies \quad k=q.$$
\noindent\textbf{Case II: \underline{$s$-odd}}\\
When $s$ is odd, RHS of Equation \eqref{e7} may be written as 
%\begin{multline}\label{e12}
   % \left|\Omega_1\right|\left[\left(q^{s-1}\left|\Omega_1\right|-1\right)\left(\frac{q^{2s-2}-1}{q^2-1}\right)-\left(\frac{q^{2 s}-1}{q^2-1}\right)\left(\frac{q^{s-1}\left(q^s+1\right)}{q+1}-1\right)\right]\\=\frac{q^s\left(q^{s-1}-1\right)}{q+1}\left[\left|\Omega_1\right|\left(\frac{q^{2 s}-1}{q^2-1}\right)-\left(\frac{q^{2 s+2}-1}{q^2-1}\right)\left(\frac{q^s+1}{q+1}\right)\right].
%\end{multline}
%Now RHS of Equation \eqref{e12} is
\begin{equation*}
    \frac{q^s\left(q^{s-1}-1\right)}{q+1}\left[\left|\Omega_1\right|\left(\frac{q^{2 s}-1}{q^2-1}\right)-\left(\frac{q^{2 s+2}-1}{q^2-1}\right)\left(\frac{q^s+1}{q+1}\right)\right],
\end{equation*}
from Lemma \ref{hyp}, putting $\displaystyle \left|\Omega_1\right|$, we obtain the RHS of Equation \eqref{e7} as
\begin{equation*}
    \frac{q^s\left(q^{s-1}-1\right)}{q+1}\left[\left(\frac{q^{s+2}+1}{q+1}-\left(q^2-q+1\right)+t\right)\left(\frac{q^{2 s}-1}{q^2-1}\right)-\left(\frac{q^{2 s+2}-1}{q^2-1}\right)\left(\frac{q^s+1}{q+1}\right)\right],
\end{equation*}
%\begin{multline*}
%= \frac{q^s\left(q^{s-1}-1\right)}{(q+1)^2(q^2-1)} \Big[ \left(q^{s+2}+1 - (q^2 - q + 1)(q+1) + t(q+1)\right)\left(q^{2s} - 1\right) \\
%- \left(q^{2s+2} - 1\right)\left(q^s + 1\right) \Big]
%\end{multline*}
equivalently,
%\begin{equation*}
   % \frac{q^s\left(q^{s-1}-1\right)}{(q+1)^2(q^2-1)}\left[t\left(q^{2s+1}+q^{2s}-q-1\right) - \left(q+1\right)\left(q^{2s+2}+q^{s+1}-q^s-q^2+q-1\right)\right].
%\end{equation*}
%on simplifying, we have
\begin{equation*}
    \frac{q^s\left(q^{s-1}-1\right)\left(q^s+1\right)}{(q+1)\left(q^2-1\right)}\left[t\left(q^s-1\right)-\left(q^2\left(q^s-1\right)+(q-1)\right)\right].
\end{equation*}
Now LHS of Equation \eqref{e7} is 
\begin{equation*}
    |\Omega_1|\left[\left(q^{s-1}|\Omega_1|-1\right)\left(\frac{q^{2s-2}-1}{q^2-1}\right)-\left(\frac{q^{2 s}-1}{q^2-1}\right)\left(\frac{q^{s-1}\left(q^s+1\right)}{q+1}-1\right)\right],
\end{equation*}
%\begin{multline*}
%= |\Omega_1| \Bigg[ \left( q^{s-1} \left( \frac{q^{s+2}+1}{q+1} - \left(q^2 - q + 1\right) + t \right) - 1 \right) \left( \frac{q^{2s-2}-1}{q^2 - 1} \right) \\
%- \left( \frac{q^{2s} - 1}{q^2 - 1} \right) \left( \frac{q^{s-1} \left(q^s + 1\right)}{q + 1} - 1 \right) \Bigg]
%\end{multline*}
from Lemma \ref{hyp}, putting $\displaystyle \left|\Omega_1\right|$, we obtain the LHS of Equation \eqref{e7} as
%\begin{multline*}
%\frac{|\Omega_1|}{(q+1)(q^2 - 1)} \Big[ \left( q^{2s+1} - q^{s+2} + tq^s + tq^{s-1} - q - 1 \right) (q^{2s+2} - 1) \\
%- (q^{2s} - 1)(q^{2s-1} + q^{s-1} - q - 1) \Big]
%\end{multline*}
\begin{equation*}
    \frac{q^{s-1}\left(q^{s-1}-1\right)}{(q+1)\left(q^2-1\right)}\left[q^3\left(q^{s-1}-1\right)+t(q+1)\right]\left[t\left(q^{s-1}+1\right)-\left(q^2\left(q^{s-1}+1\right)-(q-1)\right)\right].
\end{equation*}
Now, equating the LHS and RHS and simplifying, we get
%\begin{multline*}
    %\frac{q^{s-1}\left(q^s-1\right)}{(q+1)\left(q^2-1\right)}\left[q^3\left(q^{s-1}-1\right)+t(q+1)\right]\left[t\left(q^{s-1}+1\right)-q^2\left(q^{s-1}+1\right)+(q-1)\right] \\
%=\frac{q^s\left(q^{s-1}-1\right)\left(q^s+1\right)}{(q+1)\left(q^2-1\right)}\left[t\left(q^s-1\right)-q^2\left(q^s-1\right)-(q-1)\right]
%\end{multline*}
%\begin{multline*}
   % \left[q^3\left(q^{s-1}-1\right)+t(q+1)\right]\left[t\left(q^{s-1}+1\right)-q^2\left(q^{s-1}+1\right)+(q-1)\right] \\
 %=q\left(q^s+1\right)\left[t(q^s-1)-q^2\left(q^s-1\right)-(q-1)\right]
%\end{multline*}
%\begin{equation*}
    %t\left(q^3-q\right)-t\left(q^2-1\right)+(q+1)\left(q^{s-1}+1\right) t\left(q^2-t\right)=q\left(q^2-1\right)\left(q^s+q^2-q+1\right)
%\end{equation*}
\begin{equation}\label{e14}
    t\left(q^{s-1}+1\right)\left(q^2-t\right)=(q-1)\left[q^2\left(q^{s-1}+1\right)+(q-1)(q^2-q-t)\right]. 
\end{equation}
It follows that
\begin{equation}\label{e15}
    \left(q^{s-1}+1\right)\ \bigg | (q-1)^2(q^2-q-t).
\end{equation}
Now, we determine $t$ through case-by-case observation.
\begin{itemize}
    \item[$(a.)$] For $s =3$, by Equation \eqref{e15}, $(q^2+1)$ divides $(q-1)^2(q^2-q-t)$.
    \begin{itemize}
        \item [$(i.)$] Assume first $q$ is even. \\
    Observe that gcd($(q-1)^2,q^2+1$) = gcd($2q,q^2+1$) = $1$. As $q^2+1$ is an odd integer and the only divisors of $2q$ are of the form $2^m$ for $m \geqslant 1$, this implies $q^2+1$ divides $q^2-q-t$. Since $t$ is a non-negative integer, $t = q^2-q$.

    \item[$(ii.)$] Consider now that $q$ is odd. \\
    Observe that gcd($(q-1)^2,q^2+1 $) = gcd($q-1,q^2+1 $). Since, $(q^2+1)-(q+1)(q-1) = 2$, it follows that gcd($q-1,q^2+1 $) = gcd $(q-1,2)$. As  $q$ is odd, we have gcd $(q-1,2) = 2$. Hence gcd($(q-1)^2,q^2+1 $) = $2$. Then $\gcd\left(\displaystyle\frac{(q-1)^2}{2},\frac{q^2+1}{2}\right)$ = 1. This implies by Equation \eqref{e15},
    %Then $ q = 2m+1$, for some $m \geqslant 1$. 
	%Then gcd($(q-1)^2,q^2+1 $) = gcd($2q,q^2+1 $) = $2t$, for some odd integer $t \geqslant 1$ . Suppose that a odd prime divisor of $t$ say $d$ divides both $(q-1)^2$ and $q^2+1$. Then $d$ divides $(q-1)$, i.e., $q \cong 1 (mod~ d)$ . Then from $q^2+1$, $2 \cong 0 (mod ~d)$, i.e., $d$ divides $2$, which is not possible. %If $d=2$, then gcd($(q-1)^2,q^2+1 $) = gcd($2q,q^2+1 $) = 1.
    %Hence gcd($(q-1)^2,q^2+1 $) = $2$. Then gcd($\frac{(q-1)^2}{2},\frac{q^2+1}{2}$) = 1. This implies 
    \begin{equation*}
        \frac{q^2+1}{2} \bigg | q^2-q-t,
    \end{equation*}
      equivalently, $q^2-q-t$ = $\displaystyle r \left( \frac{q^2+1}{2}\right)$, for some $r\in \mathbb{Z}$.\\
      that is, $ t$ = $\displaystyle q^2-q-r  \left(\frac{q^2+1}{2}\right)$.\\
    As $1\leqslant t\leqslant q^2$, we have $\displaystyle 1\leqslant q^2-q-r (\frac{q^2+1}{2})\leqslant q^2$.\\
  %So,    $\displaystyle  -q\leq l (\frac{q^2+1}{2})\leq q^2-q$,\\
  equivalently,
      $\displaystyle  \frac{-2q}{q^2+1} \leqslant r \leqslant 
		\frac{2(q^2-q-1)}{q^2+1}$.
      
       The only possible integer value is $ r = 0$ or $1$. Then $ t = q^2-q$ or $\displaystyle t = q^2-q-\frac{q^2+1}{2}$. But $\displaystyle  t = q^2-q-\frac{q^2+1}{2}$, does not satisfy our equality in Equation \eqref{e14}. Hence $t = q^2-q$.
       \end{itemize}
    
    \item[$(b.)$] For $ s = 5$, $(q^4+1)$ divides $(q-1)^2(q^2-q-t)$. Now $(q-1)^2(q^2-q-t) \leqslant q^4+1$ for all $q$ and $1\leqslant t \leqslant q^2$. Hence $(q-1)^2(q^2-q-t) = 0$, i.e. $t = q^2-q$.

    \item[$(c.)$] For $s\geqslant 7$, $(q^{s-1}+1) $ is always larger than $(q-1)^2(q^2-q-t) $ as the degree of $(q^{s-1}+1) $ is at least $6$, and the degree of $(q-1)^2(q^2-q-t) $ is at most $4$. Then $(q-1)^2(q^2-q-t) = 0$. Hence $t = q^2-q$.
\end{itemize}
Therefore, for every $q$, $s$-odd, $t = q^2 - q$, and for all $q$, $s$-even, $t = q$.\\
%Thus, $t=q^2-q$ for all $q$, $s$-odd and $t=q$ for all $q$, $s$-even.\\
Then, from Lemma \ref{hyp}, we get
%\begin{equation*}
    %\left|\Omega_1\right|= \begin{cases}\frac{q^{s+2}-1}{q+1}+1=\frac{q\left(q^{s+1}+1\right)}{q+1}, & n \text {-even } \\
%\frac{q^{s+2}+1}{q+1}-1=\frac{q\left(q^{s+1}-1\right)}{q+1}, & n \text {-odd }\end{cases}
%\end{equation*}
\begin{equation*}
    \left|\Omega_1\right|= \frac{q\left(q^{s+1}+(-1)^s\right)}{q+1},
    %q^{s-1}\left|\Omega_1\right|= \begin{cases}q^s \frac{\left(q^{s+1}+1\right)}{q+1}  \quad\text {, $s$-even } \\
%\frac{q^s\left(q^{s+1}-1\right)}{q+1}, \quad\text { $s$-odd }
%\end{cases} 
\end{equation*}
and hence
\begin{equation*}
   |\Omega|=\frac{q^s\left(q^{s+1}+(-1)^s\right)}{q+1}.
\end{equation*}
\end{proof}

%==========================================================================

\begin{cor} \label{black-pg}
   The number of black points in $\mathrm{P G}\left(s, q^2\right)$ is $\left|\mathrm{\mathcal{H}}\left(s, q^2\right)\right|$.
   %$\displaystyle \frac{\left(q^{s+1}+(-1)^s\right)\left(q^s-(-1)^s\right)}{q^2-1}$.
\end{cor}

\begin{proof}
By substituting the value of $|\Omega|$ into Equation \eqref{e3}, we get %the following.
%Putting the value of $|\Omega|$ in Equation \eqref{e4}, we get
\begin{equation*}
    b  =(-1)^s\left[q\left(\frac{q^{s+1}+(-1)^s}{q+1}\right)\left(\frac{q^{2 s}-1}{q^2-1}\right)-\left(\frac{q^{2 s+2}-1}{q^2-1}\right)\left(\frac{q^s-(-1)^s}{q+1}\right)\right],
\end{equation*}
%\begin{equation*}
    %=(-1)^s\left(\frac{q^{s+1}+(-1)^s}{q+1}\right)\left(\frac{q^s-(-1)^s}{q^2-1}\right)\left[q\left(q^s+(-1)^s\right)-\left(q^{s+1}-(-1)^s\right)\right]
%\end{equation*}
after simplifying, we obtain
\begin{equation*}
    b  = \left|\mathrm{\mathcal{H}}\left(s, q^2\right)\right|.
    %\frac{\left(q^{s+1}+(-1)^s\right)\left(q^s-(-1)^s\right)}{q^2-1}.
\end{equation*}
\end{proof}

%===========================================================================

\begin{cor} \label{black-hyp}
     A hyperplane of $\Omega$ contains $\left|\mathrm{\mathcal{H}}\left(s-1, q^2\right)\right|$ black points.
     %$\displaystyle \frac{\left(q^s-(-1)^s\right)\left(q^{s-1}+(-1)^s\right)}{q^2-1}$ black points.
\end{cor}

\begin{proof}
Let $\Phi$ be any hyperplane of $\Omega$.
Then from Lemma \ref{bl-hyp},
\begin{equation*}
    b_\Phi=\frac{(-1)^s}{q^{s-1}}\left[\left(\frac{q^s\left(q^{s+1}+(-1)^s\right)}{q+1}-1\right)\left(\frac{q^{2 s-2}-1}{q^2-1}\right)-\left(\frac{q^{2 s}-1}{q^2-1}\right)\left(\frac{q^{s-1}\left(q^s-(-1)^s\right)}{q+1}-1\right)\right],
\end{equation*}
%\begin{equation*}
% = \frac{(-1)^s}{q^{s-1}(q+1)(q^2-1)}\bigg[\big(q^{2 s+1}+(-1)^s q^s-q-1\big)\big(q^{2 s-2}-1\big)\\
    %-\big(q^{2 s}-1\big)\big(q^{2 s-1}-{(-1)}^n q^{s-1}-q-1\big)\bigg]
%\end{equation*}
after simplifying, we obtain
\begin{equation*}
    b_\Phi= \left|\mathrm{\mathcal{H}}\left(s-1, q^2\right)\right|.
    %\frac{\left(q^s-(-1)^s\right)\left(q^{s-1}+(-1)^s\right)}{q^2-1}.
\end{equation*}
\end{proof}

%========================================================================

\begin{Lemma}\label{black-hyp not}
     {\em A hyperplane not in $\Omega$ contains $\left| p\mathrm{\mathcal{H}}\left(s-2, q^2\right)\right|$
     %$\displaystyle 1+\frac{q^2\left(q^{s-1}+(-1)^s\right)\left(q^{s-2}-(-1)^s\right)}{q^2-1}$
     black points.}
\end{Lemma}

\begin{proof}
    Let $\Pi$ be a hyperplane not in $\Omega$. Denote $b_\Pi$  as number of black points in $\Pi$. Then there are $\displaystyle \frac{q^{2 s}-1}{q^2-1}- b_\Pi$ white points in $\Pi$.\\
     %and $w_\Pi$ as number of black and white points in $\Pi$, respectively. Then\\
%\begin{equation} \label{e18}
   % b_\Pi+w_\Pi=\frac{q^{2 s}-1}{q^2-1}
%\end{equation}
Counting the given set 
\begin{equation*}
   \{(x, \Psi) \mid x \text{~ is a point of~} \mathrm{PG}\left(s, q^2\right), \Psi \not \in \Omega,~ \Psi \neq \Pi ~\text{and} ~x \in \Psi \cap\Pi\},
\end{equation*}
in two ways, we obtain
\begin{multline*}\label{e19}
    b_\Pi\left[\frac{q^{2 s}-1}{q^2-1}-\frac{q^s\left(q^{s-1}-(-1)^{s-1}\right)}{q+1}-1\right]+\left(\frac{q^{2 s}-1}{q^2-1}- b_\Pi \right)\left[\frac{q^{2 s}-1}{q^2-1}-\frac{q^{s-1}\left(q^s-(-1)^s\right)}{q+1}-1\right] \\
    =\left(\frac{q^{2 s+2}-1}{q^2-1}-|\Omega|-1\right)\left(\frac{q^{2 s-2}-1}{q^2-1}\right).
\end{multline*}
Now putting the value of $|\Omega|$ and solving the above equation for $b_\Pi$, we get 
\begin{equation*}
     b_\Pi = \left| p\mathrm{\mathcal{H}}\left(s-2, q^2\right)\right|.
     %1+\frac{q^2\left(q^{s-1}+(-1)^{s-2}\right)\left(q^{s-2}-(-1)^{s-2}\right)}{q^2-1}=1+\frac{q^2\left(q^{s-1}+(-1)^s\right)\left(q^{s-2}-(-1)^s\right)}{q^2-1}.
 \end{equation*}
\end{proof}

%============================================================================

\begin{Lemma}\label{codim}
    {\em Every co-dimension $2$ subspace of $\mathrm{P G}\left(s, q^2\right)$ contains $\left|\mathcal H\left(s-2, q^2\right)\right|, \left| p\mathrm{\mathcal{H}}\left(s-3, q^2\right)\right|$ or $\left|\mathrm{L \mathcal{H}}\left(s-4, q^2\right)\right|$ black points.}
\end{Lemma}

\begin{proof}
    Let $\Psi$ be a co-dimension 2 subspace of $\mathrm{P G}\left(s, q^2\right)$ and denote $b_\Psi$ as the number of black points in $\Psi$.\\
Let us denote $k$ as the number of hyperplanes of $\Omega$ containing $\Psi$. Now, counting the total number of black points in $\mathrm{PG}  (s,q^2)$ by considering all the hyperplanes through $\Psi$
%\begin{multline*}
   %S_{i}\left[\frac{\left(q^s-(-1)^s\right)\left(q^{s-1}+(-1)^s\right)}{q^2-1}-b_\Psi\right]+\left(q^2+1-S_i\right)\left[1+q^2 \frac{\left(q^{s-1}+(-1)^s\right)\left(q^{s-2}-(-1)^s\right)}{q^2-1}-b_\Psi\right]\\+b_\Psi  
% =\frac{\left(q^{s+1}+(-1)^s\right)\left(q^s-(-1)^s\right)}{q^2-1} , 
%\end{multline*}
%or equivalently 
    \begin{equation*}
      k \left|\mathrm{\mathcal{H}}\left(s-1, q^2\right)\right|+\left(q^2+1-k\right) \left|p\mathrm{\mathcal{H}}\left(s-2, q^2\right)\right| -q^2 b_\Psi=\left|\mathrm{\mathcal{H}}\left(s, q^2\right)\right|.
      %\left[\frac{\left(q^s-(-1)^s\right)\left(q^{s-1}+(-1)^s\right)}{q^2-1}\right]+\left(q^2+1-S_i\right)\left[1+\frac{q^2\left(q^{s-1}+(-1)^s\right)\left(q^{s-2}-(-1)^s\right)}{q^2-1}\right]-q^2 b_\Psi \\
%= \left|\mathrm{\mathcal{H}}\left(s, q^2\right)\right|
%\frac{\left(q^{s+1}+(-1)^s\right)\left(q^s-(-1)^s\right)}{q^2-1}.
    \end{equation*}
    By simplifying, we have 
\begin{equation}\label{e20}
        (-1)^sq^{s-1}k+\left(q^2+1\right) \left|p\mathrm{\mathcal{H}}\left(s-2, q^2\right)\right| =\left|\mathrm{\mathcal{H}}\left(s, q^2\right)\right| +q^2 b_\Psi.
        %\left[1+\frac{q^2\left(q^{s-1}+(-1)^s\right)\left(q^{s-2}-(-1)^s\right)}{q^2-1}\right] \\
%=\frac{\left(q^{s+1}+(-1)^s\right)\left(q^s-(-1)^s\right)}{q^2-1}+q^2 b_\Psi.
    \end{equation}
    For $k=0$, $ b_\Psi=\left|\mathrm{L \mathcal{H}}\left(s-4, q^2\right)\right|$.
    If $k \neq 0$, then by Theorem 1.1(II), $q^2-q \leqslant k \leqslant q^2.$\\ Then from Equation \eqref{e20}, we have $\left|\mathrm{\mathcal{H}}\left(s-2, q^2\right)\right| \leqslant b_\Psi \leqslant \left|p\mathrm{\mathcal{H}}\left(s-3, q^2\right)\right|$. %or \\ $\left|p\mathrm{\mathcal{H}}\left(s-3, q^2\right)\right| \leqslant b_\Psi \leqslant \left|\mathrm{\mathcal{H}}\left(s-2, q^2\right)\right|$ depending on $s$ is even or odd, respectively. \\
    Now we show that $b_\Psi = \left|\mathrm{\mathcal{H}}\left(s-2, q^2\right)\right|$ or $b_\Psi = \left|p\mathrm{\mathcal{H}}\left(s-3, q^2\right)\right|$.\\

     Consider a fixed hyperplane $\Phi \in \Omega$. Let $b_{\Psi_i}$ represent the number of black points in $\Psi_i$, where $\Psi_i$ are co-dimension $2$ space contained in $\Phi$. Now counting the following incident pairs and triples,
     \begin{equation*}
         \{(x,\Psi)~|~ x \text{~is a black point in $\Phi$}, \Psi~ \text{is a co-dimension $2$ subspace of} ~\Phi ~ \text{and} ~ x \in \Psi\}
     \end{equation*}
     and
     \begin{equation*}
         \{(x,y, \Psi)~|~ x, y \text{~ black points in }\Phi, \Psi \in \Phi~\text{and}~ x \neq y \in \Psi \},
     \end{equation*}
     we obtain
     \begin{equation*}
         \sum b_{\Psi_i}= \left|\mathrm{\mathcal{H}}\left(s-1, q^2\right)\right| \frac{q^{2s-2}-1}{q^2-1},
     \end{equation*}
     and
     \begin{equation*}
         \sum b_{\Psi_i}(b_{\Psi_i}-1)= \left|\mathrm{\mathcal{H}}\left(s-1, q^2\right)\right| \left(\left|\mathrm{\mathcal{H}}\left(s-1, q^2\right)\right| -1\right) \frac{q^{2s-4}-1}{q^2-1}.
     \end{equation*} 
     Now we have 
     \begin{equation*}
         \sum \left(b_{\Psi_i} - \left|\mathrm{\mathcal{H}}\left(s-2, q^2\right)\right| \right)\left(b_{\Psi_i} - \left|p\mathrm{\mathcal{H}}\left(s-3, q^2\right)\right|\right) = 0.
     \end{equation*}
     Since, $\left|\mathrm{\mathcal{H}}\left(s-2, q^2\right)\right| \leqslant b_\Psi \leqslant \left| p\mathrm{ \mathcal{H}}\left(s-3, q^2\right) \right|$ %or $\left|p\mathrm{\mathcal{H}}\left(s-3, q^2\right)\right| \leqslant b_\Psi \leqslant \left|\mathrm{\mathcal{H}}\left(s-2, q^2\right)\right|$
     , we obtain that $b_\Psi = \left|\mathrm{\mathcal{H}}\left(s-2, q^2\right)\right|$ or $b_\Psi = ~\left|p\mathrm{\mathcal{H}}\left(s-3, q^2\right)\right|$.\\
     Hence, every co-dimension $2$ subspace of $\mathrm{P G}\left(s, q^2\right)$ contains $\left|\mathcal H\left(s-2, q^2\right)\right|, \left| p\mathrm{\mathcal{H}}\left(s-3, q^2\right)\right|$ or $\left|\mathrm{L \mathcal{H}}\left(s-4, q^2\right)\right|$ black points.
      \end{proof}
     \vspace{1cm}

    \noindent $\mathbf{Proof~ of ~Theorem~ \ref{main thm}}$
    From Corollary \ref{black-hyp} and Lemma \ref{black-hyp not}, every hyperplane of $\mathrm{P G}\left(s, q^2\right)$ contains $\left|\mathrm{\mathcal{H}}\left(s-1, q^2\right)\right|$ or $\left| p\mathrm{\mathcal{H}}\left(s-2, q^2\right)\right|$
    %$\displaystyle \frac{\left(q^s-(-1)^s\right)\left(q^{s-1}+(-1)^s\right)}{q^2-1}$ or $\displaystyle 1+\frac{q^2\left(q^{s-1}+(-1)^s\right)\left(q^{s-2}-(-1)^s\right)}{q^2-1}$ 
    black points and by Lemma \ref{codim}, every co-dimension $2$ subspace of $\mathrm{P G}\left(s, q^2\right)$ contains $\left|\mathcal H\left(s-2, q^2\right)\right|, \left| p\mathrm{\mathcal{H}}\left(s-3, q^2\right)\right|$ or $\left|\mathrm{L \mathcal{H}}\left(s-4, q^2\right)\right|$ black points. Hence, by Proposition \ref{use}, the collection of black points constitute a non-singular hermitian variety $\mathcal{H}\left(s, q^2\right)$.
  % \end{proof} 
   
    \vspace{1cm}
    $\mathbf{Acknowledgement}$: The first author is supported by University Grants Commission(UGC) Junior Research Fellowship(JRF), NTA Ref.No: 211610115326. The second author is supported by project no. SRG/2021/000603 of Science and Engineering Research Board (SERB), Department of Science and Technology, Government of India.

\vspace{3cm}
    \noindent $\mathbf{Address:}$\\
    Stuti Mohanty (Email: stutimohanty034@gmail.com)\\
    Bikramaditya Sahu (Email: sahuba@nitrkl.ac.in)\\
    Department of Mathematics, National Institute of Technology\\
    Rourkela - 769008, Odisha, India.

\end{document}